\documentclass[12pt]{amsart}
\usepackage{geometry}
\usepackage[colorlinks, citecolor = red, linkcolor=blue, hyperindex]{hyperref}
\usepackage{euscript, eufrak, verbatim, mathrsfs}
\usepackage[psamsfonts]{amssymb}
\usepackage{bbm}
\usepackage{graphicx}
\usepackage{float}
\usepackage{float, tikz}
\usepackage[all, cmtip]{xy}
\usepackage{upref, xcolor, dsfont}
\usepackage{amsfonts, amsmath, amstext, amsbsy, amsopn, amsthm}
\usepackage{enumerate}
\usepackage{url}
\usepackage{mathtools}
\usepackage{bookmark}
\usepackage{euscript}
\usepackage{helvet}         
\usepackage{courier}        
\usepackage{type1cm}        
\usepackage{multicol}        
\usepackage[bottom]{footmisc}

\newtheorem{theorem}{Theorem}[section]
\newtheorem*{theorem*}{Theorem}
\newtheorem{lemma}[theorem]{Lemma}

\newtheorem{claim}[theorem]{Claim}
\newtheorem{proposition}[theorem]{Proposition}

\newtheorem{aomment}[theorem]{Comment}
\newtheorem*{comment*}{Comment}
\newtheorem*{definition*}{Definition}
\newtheorem*{remark*}{Remark}

\newtheorem*{observation*}{Observation}

\newtheorem*{assumption*}{Assumption}

\theoremstyle{definition}

\theoremstyle{remark}

\newcommand{\PP}{\mathbb{P}}

\newcommand{\Tr}{\mathrm{Tr}}

\newcommand{\Var}{\mathrm{Var}}

\newcommand{\Conf}{\mathrm{Conf}}

\begin{document}

\title[Gaussian limit for DPPs with $J$-Hermitian kernels]{Gaussian limit for determinantal point processes with $J$-Hermitian kernels}

\author
{Zhaofeng Lin}
\address{Zhaofeng Lin: Shanghai Center for Mathematical Sciences, Fudan University, Shanghai, 200438, China.}
\email{zflin18@fudan.edu.cn}

\author
{Yanqi Qiu}
\address
{Yanqi QIU: Institute of Mathematics and Hua Loo-Keng Key Laboratory of Mathematics, AMSS, Chinese Academy of Sciences, Beijing 100190, China.}
\email{yanqi.qiu@amss.ac.cn; yanqi.qiu@hotmail.com}

\author
{Kai Wang}
\address{Kai WANG: School of Mathematical Sciences, Fudan University, Shanghai, 200433, China.}
\email{kwang@fudan.edu.cn}

\thanks{Y. Qiu is supported by grants NSFC Y7116335K1, NSFC 11801547 and NSFC 11688101 of National Natural Science Foundation of China. K. Wang is supported by grants NSFC 11722102, NSFC 12026250 and LMNS, Fudan.}

\begin{abstract}
We show that the central limit theorem for linear statistics over determinantal point processes with $J$-Hermitian kernels holds under fairly general conditions. In particular, We establish Gaussian limit for linear statistics over determinantal point processes on union of two copies of $\mathbb{R}^d$ when the correlation kernels are $J$-Hermitian translation-invariant.
\end{abstract}



\maketitle

\section{Introduction and main results}
\subsection{Determinantal point processes}
Let $X$ be a locally compact Polish space, $\mathcal{B}(X)$ be the Borel $\sigma$-algebra on $X$, and denote $\mathcal{B}_0(X)$ the collection of all pre-compact Borel sets. Let $\Conf(X)$ denote the space of all locally finite configurations over $X$, that is,
\begin{equation*}
	\begin{split}
        \Conf(X):=\big\{\xi=\textstyle\sum_{i}\delta_{x_i}\,\big|\,\,\forall i,\,x_i\in X\, \text{and $\xi(\Delta)<\infty$ for all $\Delta\in\mathcal{B}_0(X)$}\big\}.
	\end{split}
\end{equation*}
Consider the vague topology on $\Conf(X)$, the weakest topology on $\Conf(X)$ such that all maps $\Conf(X)\ni\xi\mapsto\int_{X}fd\xi$, $f\in C_c(X)$, are continuous. Here $C_c(X)$ is the space of all continuous real-valued functions on $X$ with compact support. The configuration space $\Conf(X)$ equipped with the vague topology becomes a Polish space. The Borel $\sigma$-algebra $\mathcal{F}$ on $\Conf(X)$ is generated by the cylinder sets $C_n^\Delta=\big\{\xi\in\Conf(X)\,|\,\,\xi(\Delta)=n\big\}$, where $\Delta\in\mathcal{B}_0(X)$ and $n\in\mathbb{N} = \{0,1,2,\cdots\}$. By definition,  a Borel probability measure $\mathbb{P}$ on $\Conf(X)$ is called a point process on $X$. For more details, see, e.g. \cite{DV,KK,Le}.

Let $\mu$ be a reference Radon measure on $X$. The $n$-th ($n\in\mathbb{N_+}$) correlation function with respect to $\mu$ of a point process $\mathbb{P}$  on $X$, if exists, is a $\mu^{\otimes n}$-a.e. non-negative measurable symmetric function $\rho_n:X^n\to [0, +\infty)$, such that for any family of mutually disjoint subsets $\Delta_1,\Delta_2,\cdots,\Delta_n\in\mathcal{B}_0(X)$, $m\geq 1, n_i\geq 1, n_1+n_2+\cdots+n_m = n$,
\begin{equation}\label{def-corre-func}
	\begin{split}
        \mathbb{E}_{\mathbb{P}}\Big[\prod_{i=1}^{n}\frac{(\#_{\Delta_i})!}{(\#_{\Delta_i}-n_i)!}\Big]=\int_{\Delta_1^{n_1}\times\cdots\times\Delta_m^{n_m}} \rho_n(x_1,\cdots,x_n)d\mu(x_1)\cdots d\mu(x_n),
	\end{split}
\end{equation}
where $\#_\Delta:\Conf(X)\to\mathbb{N}$ is defined by $\#_\Delta(\xi):=\xi(\Delta)$.

A configuration $\xi=\textstyle\sum_{i}\delta_{x_i}\in\Conf(X)$ is called simple if $\xi(\{x\}) \le 1$ for all $x\in X$. When $\xi$ is simple, we identify $\xi=\textstyle\sum_{i}\delta_{x_i}$ with the subset $\{x_i\}\subset X$. A point process $\mathbb{P}$  on $X$ is called simple if $\PP$-almost every $\xi\in\Conf(X)$ is simple.

Let $K:X\times X\to\mathbb{C}$ be a measurable function, by slightly abusing the notation, we also denote $K$ the associated integral operator with integral kernel $K(x,y)$, that is, $K:L^2(X,\mu)\to L^2(X,\mu)$ is defined by 
\begin{equation}\label{associ-opera}
	\begin{split}
		Kf(x)=\int_X K(x,y)f(y)d\mu (y)\,,\quad f\in L^2(X,\mu).
	\end{split}
\end{equation}
A simple point process is called determinantal with kernel $K$ with respect to $\mu$ if it admits correlation functions of all orders, and
\begin{equation}\label{def-dpp}
	\begin{split}
        \rho_n(x_1,\cdots,x_n)=\det\big[K(x_i,x_j)\big]_{1\leq i, j\leq n}
	\end{split}
\end{equation}
for every $n\geq 1$ and $\mu^{\otimes n}$-a.e. $(x_1,\cdots ,x_n)\in X^n$. We refer the reader to \cite{Bo,HKPV,PV,ST,So} for further background and details of determinantal point processes.

\subsection{DPPs with $J$-Hermitian kernels}
Suppose that the underlying space $X$ is split into two disjoint parts $X=X_1\sqcup  X_2$ with $X_1,X_2\in\mathcal{B}(X)$. Hence we have the direct sum decomposition $L^2(X,\mu)=L^2(X_1,\mu)\oplus L^2(X_2,\mu)$, and for $i=1,2$, we denote $P_i$ the orthogonal projection of $L^2(X,\mu)$ onto $L^2(X_i,\mu)$. Let $J:=P_1-P_2$, and then we define a $J$-scalar product on $L^2(X,\mu)$ by
\begin{equation*}
	\begin{split}
        [f,g]:=\left\langle Jf,g\right\rangle=\left\langle P_1f,P_1g\right\rangle-\left\langle P_2f,P_2g\right\rangle,\quad f,g\in L^2(X,\mu),
	\end{split}
\end{equation*}  
here $\left\langle\cdot,\cdot\right\rangle$ is the usual scalar product in $L^2(X,\mu)$. 

A bounded linear operator $K$ on $L^2(X,\mu)$ is called $J$-self-adjoint if $[Kf,g]=[f,Kg]$ for all $f,g\in L^2(X,\mu)$. A $J$-self-adjoint operator $K$ can be expressed in block form: 
\begin{equation*}
		K=
		\begin{bmatrix}
			K_{11}&K_{12}\\
			-K_{12}^*&K_{22}
		\end{bmatrix}
\end{equation*}
satisfying
\begin{equation}\label{J-Hermitian}
	\begin{split}
		\left\{\begin{array}{l}
			P_1K^*P_1=P_1KP_1\\
			P_1K^*P_2=-P_1KP_2\\
			P_2K^*P_1=-P_2KP_1\\
			P_2K^*P_2=P_2KP_2
		\end{array}\right..
	\end{split}
\end{equation}
A measurable function $K(x,y)$ is called $J$-Hermitian if it induces a bounded $J$-self-adjoint integral operator by \eqref{associ-opera}. Note that $K(x,y)$ is $J$-Hermitian iff for $\mu\otimes\mu$-a.e. $(x,y)\in X^2$,
\begin{equation}\label{J-Hermitian-kernel}
	\begin{split}
		\left\{\begin{array}{l}
			K(x,y)=\overline{K(y,x)}\,,\quad\,x,y\in X_1\,\,\text{or}\,\,x,y\in X_2\,,\\
			K(x,y)=-\overline{K(y,x)}\,,\quad\,x\in X_1$, $y\in X_2\,.
		\end{array}\right.
	\end{split}
\end{equation}
We denote an operator $\widehat K$ through $P_1$, $P_2$ and $K$ by
\begin{equation}\label{def-Khat}
        \widehat K:=KP_1+(1-K)P_2,
\end{equation}  
then the operator $\widehat K$ has the following block form:
\begin{equation*}
		\widehat K=
		\begin{bmatrix}
			K_{11}&K_{12}\\
			K_{12}^*&P_2-K_{22}
		\end{bmatrix}.
\end{equation*}  
It is easy to see that $K$ is $J$-self-adjoint if and only if $\widehat{K}$ is self-adjoint. We refer the reader to \cite{BOls,BOlsh,BOl,BOlsha,BOO,BQ,Ly} for more details and examples of $J$-Hermitian kernels. In particular, the matrix tail kernel derived by Borodin and Olshanski  in \cite[Theorem~VII]{BOl} is a $J$-Hermitian translation-invariant kernel on $\mathbb{R}\sqcup\mathbb{R}$.

We denote by $\mathscr{L}_{1|2}(L^2(X,\mu))$ the set of all bounded linear operators $A$ on $L^2(X,\mu)$ such that $P_1AP_1+P_2AP_2$ is trace-class and $P_1AP_2+P_2AP_1$ is Hilbert-Schmidt. For $A\in\mathscr{L}_{1|2}(L^2(X,\mu))$, we set 
\begin{equation}\label{def-Tr}
	\begin{split}
        \Tr(A):=\Tr(P_1AP_1+P_2AP_2).
	\end{split}
\end{equation} 
Note that  when $n\geq2$, for any $A_1,A_2,\cdots,A_n\in\mathscr{L}_{1|2}(L^2(X,\mu))$, since  $A_1,A_2,\cdots,A_n$ are Hilbert-Schmidt operators, $A_1A_2\cdots A_n$ is trace-class, hence we may define $\Tr(A_1A_2\cdots A_n)$ as  the usual trace of $A_1A_2\cdots A_n$.

Suppose that $K$ is a $J$-self-adjont, locally $\mathscr{L}_{1|2}(L^2(X,\mu))$, bounded linear operator on $L^2(X,\mu)$. Here $K$ is locally $\mathscr {L}_{1|2}(L^2(X,\mu))$ means that for each $\Delta\in\mathcal{B}_0(X)$, the operator $\chi_{\Delta}K\chi _{\Delta}$ belongs to $\mathscr{L}_{1|2}(L^2(X,\mu))$, where $\chi_{\Delta}$ denotes the multiplication operator by the indicator of $\Delta$. Then by \eqref{def-Tr}, $\Tr(\chi_{\Delta}K\chi_{\Delta})$ makes sense. We can choose an integral kernel $K(x,y)$ of $K$ such that 
\begin{equation*}
	\begin{split}
		\Tr(\chi_{\Delta}K\chi_{\Delta})=\int_{\Delta}K(x,x)d\mu(x)
	\end{split}
\end{equation*}
for any $\Delta\in\mathcal{B}_0(X)$, see \cite[Proposition~13]{Ly}.  Lytvynov \cite[Theorem 3]{Ly} proves that there is a determinantal point process on $X$ with $J$-Hermitian kernel $K$ with respect to the Radon measure $\mu$ if and only if 
\begin{equation}\label{0<Khat<1}
	\begin{split}
		0\leq \widehat K\leq 1.
	\end{split}
\end{equation}
Following Soshnikov \cite[Theorem~1]{Soshn}, in what follows,  we always assume that we can choose a $J$-Hermitian integral kernel $K(x,y)$ such that for any $\Delta\in\mathcal{B}_0(X)$, the operator $K\chi_{\Delta}$ belongs to $\mathscr {L}_{1|2}(L^2(X,\mu))$, then for any $n\in \mathbb{N}_+$ and $\Delta_1,\Delta_2\cdots,\Delta_n\in\mathcal{B}_0(X)$,
\begin{equation*}
	\begin{split}
		\Tr(&K\chi _{\Delta_1}K\chi _{\Delta_2}\cdots K\chi _{\Delta_n})\\
		&=\int_{\Delta_1\times\Delta_2\times\cdots\times\Delta_n}K(x_1,x_2)K(x_2,x_3)\cdots K(x_n,x_1)d\mu (x_1)\cdots d\mu (x_n).
	\end{split}
\end{equation*}
And then a simple function approximation yields that for any $n\in \mathbb{N}_+$ and $f_1,f_2,\cdots,f_n\in B_c(X)$,
\begin{equation}\label{trace-integral}
	\begin{split}
		\Tr(&Kf_1Kf_2\cdots Kf_n)\\
		&=\int_{X^n}f_1(x_1)K(x_1,x_2)f_2(x_2)K(x_2,x_3)\cdots f_n(x_n)K(x_n,x_1)d\mu (x_1)\cdots d\mu (x_n).
	\end{split}
\end{equation}

\subsection{Gaussian limit for DPPs with Hermitian kernels}

Let us briefly recall the work of Soshnikov on the Gaussian limit of DPPs with Hermitian kernels.

Let $\mathbb{P}$ be a determinantal point process on a locally compact Polish space $X$ with kernel $K$ with respect to Radon measure $\mu$. It is important to study the behavior of linear statistics
\begin{equation*}
	\begin{split}
		S_f(\xi):=\sum_{x\in\xi}f(x)\,,\quad\xi\in\Conf(X),
	\end{split}
\end{equation*}
for test functions in a scaling limit, a lot of results have been obtained in  \cite{Ba,BW,CL,DE,DS,Joh,Jo,Sosh,Soshn,Sos,Sp,Wi}. The mathematical expectation and variance of $S_f$ can be calculated by \eqref{def-corre-func} and \eqref{def-dpp},
\begin{equation}\label{expecsf}
	\begin{split}
		\mathbb{E}_\mathbb{P}S_f=\int_Xf(x)K(x,x)d\mu(x),
	\end{split}
\end{equation}
\begin{equation}\label{varsf}
	\begin{split}
		\Var_\mathbb{P}S_f=\int_Xf^2(x)K(x,x)d\mu(x)-\int_{X^2}f(x)f(y)K(x,y)K(y,x)d\mu(x)d\mu(y).
	\end{split}
\end{equation}

Suppose that integral operator $K$ is self-adjont, non-negative definite, locally trace-class and bounded on $L^2(X,\mu)$. By a theorem obtained by Macchi \cite{Ma} and Soshnikov \cite{So}, as well as Shirai and Takahashi \cite{ST}, the integral kernel $K(x,y)$ of the operator $K$ is the correlation kernel of a determinantal point process if and only if $0\leq K\leq1$. 

For $L\geq0$, let $\mathbb{P}_L$ be a determinantal point process on locally compact Polish space $X_L$ with Hermitian kernel $K_L$ with respect to Radon measure $\mu_L$. And suppose that for any $\Delta\in\mathcal{B}_0(X_L)$, the operator $K_L\chi_{\Delta}$ is trace-class. Let $f_L\in B_c(X_L)$, the space of bounded measurable real-valued functions on $X_L$ with compact support. The work of Soshnikov \cite[Theorem~1]{Soshn} gives that,  under some mild conditions, the centered normalized linear statistics 
\begin{equation*}
	\begin{split}
		\frac{S_{f_L}-\mathbb{E}_LS_{f_L}}{\sqrt{{\Var_L}S_{f_L}}}
	\end{split}
\end{equation*}        
converges in distribution to the standard normal law $N(0,1)$ as $L\to +\infty$. From now on, we always denote $\mathbb{E}_L$ and $\Var_L$ the mathematical expectation and variance with respect to $\mathbb{P}_L$.

\subsection{Main results}
For a family of determinantal point processes with $J$-Hermitian kernels, our main result gives the central limit theorem for linear statistics when the variance grows faster than some fixed sufficiently small power of the mathematical expectation.

For $L\geq 0$, suppose that $\mathbb{P}_L$ is a determinantal point process on locally compact Polish space $X_L=X_L^{(1)}\sqcup X_L^{(2)}$ with $J$-Hermitian kernel $K_L$ with respect to Radon measure $\mu _L$. And suppose that for any $\Delta\in\mathcal{B}_0(X_L)$, the operator $K_L\chi_{\Delta}$ belongs to $\mathscr {L}_{1|2}(L^2(X_L,\mu_L))$. For $f_L\in B_c(X_L)$, consider the linear statistics 
\[
S_{f_L}(\xi)=\sum_{x\in\xi}f_L(x), \, \xi\in\Conf(X_L).
\] 

\begin{theorem}\label{thm-1}
   Suppose that 
\begin{itemize}
\item $\Var_L(S_{f_L})\to +\infty$ as $L\to +\infty$;
\item $\sup_{x\in X_L}|f_L(x)|=o((\Var_L(S_{f_L}))^\varepsilon)$ as $L\to +\infty$ for any $\varepsilon>0$;
\item $\mathbb{E}_LS_{|f_L|}=O((\Var_L(S_{f_L}))^\delta)$ as $L\to +\infty$ for some $\delta>0$.
\end{itemize}
 Then the centered normalized linear statistics
    \begin{equation*}
    	\begin{split}
    		\frac{S_{f_L}-\mathbb{E}_LS_{f_L}}{\sqrt{{\Var_L}S_{f_L}}}
    	\end{split}
    \end{equation*} 
    converges in distribution to the standard normal law $N(0,1)$ as $L\to +\infty$.
\end{theorem}

\begin{aomment}
	Theorem~\ref{thm-1} extends the work of Soshnikov for linear statistics over determinantal point processes with Hermitian kernels. Soshnikov's proof in \cite[Theorem~1]{Soshn} relies on the fact that when bounded linear self-adjoint operator $K$ is positive definite, terms like $\Tr(Kf_1Kf_2)$ are non-negative for non-negative functions $f_1$, $f_2$ with compact support, which is not the case for $J$-Hermitian kernels. This can not applied directly to $J$-Hermitian situation. 

On the other hand, when the underlying space $X$ is discrete, under the particle-hole duality transformation
\[
\Conf(X)\ni\xi\mapsto\widehat{\xi}:=(\xi\cap X_1)\cup(X_2\backslash\,\xi),
\] a $J$-Hermitian DPP turns into a Hermitian DPP and vise versa. By this particle-hole duality, it is easy to see that in the discrete setting,  Theorem~\ref{thm-1} follows from the work of Soshnikov, see \cite{BOO,BQ,Ly,Soshn}.
\end{aomment}

Next we consider $X=\mathbb{R}^d_1\sqcup\mathbb{R}_2^d:=\mathbb{R}^d\sqcup\mathbb{R}^d$, $d\mu=dx_1+dx_2$, where $dx_i$ is the Lebesgue measure on $\mathbb{R}_i^d$, $i=1,2$. Let $A$ be a $J$-self-adjoint integral operator on $L^2(X,\mu)$ with $J$-Hermitian translation-invariant integral kernel
\begin{equation}\label{J-HT-IIK}
	\begin{split}
        A(x_i,y_j)=
        \begin{bmatrix}
	        F(x_1-y_1)&G(x_1-y_2) \\ 
	        -\overline G(y_1-x_2)&H(x_2-y_2) 
        \end{bmatrix},
	\end{split}
\end{equation}
where $F,G,H\in L^2(\mathbb{R}^d)$. Denote $\mathscr{F}$ the Fourier transform on $L^2(\mathbb{R}^d)$, that is, for any $T\in L^2(\mathbb{R}^d)$,
	\begin{equation*}
	\begin{split}
		(\mathscr{F}T)(x)=\lim\limits_{R\to+\infty}\int_{|t|<R}T(t)e^{-2\pi it\cdot x}dt,
	\end{split}
\end{equation*}
where the limit is in $L^2(\mathbb{R}^d)$.
The following proposition gives a criterion to judge when does a $J$-Hermitian translation-invariant integral kernel can be a correlation kernel of a determinantal point process. 

\begin{proposition}\label{prop-2} 
	There exists a determinantal point process on $\mathbb{R}^d\sqcup \mathbb{R}^d$ with $J$-Hermitian translation-invariant kernel $A$ as in \eqref{J-HT-IIK} if and only if for almost every $x\in \mathbb{R}^d$,
	\begin{equation*}
		\begin{split}
		 \left\{\begin{array}{l}
		 	0\leq (\mathscr{F}F)(x)\leq 1\\
		 	0\leq (\mathscr{F}H)(x)\leq 1\\
		 	|(\mathscr{F}G)(x)|^2\leq (\mathscr{F}F)(x)[1-(\mathscr{F}H)(x)]\\
		 	|(\mathscr{F}G)(x)|^2\leq (\mathscr{F}H)(x)[1-(\mathscr{F}F)(x)]
		 \end{array}\right..
		\end{split}
	\end{equation*}
\end{proposition}

For $L\geq0$, let $\mathbb{P}_L$ be a determinantal point process on $X=\mathbb{R}_1^d\sqcup\mathbb{R}_2^d=\mathbb{R}^d\sqcup\mathbb{R}^d$ with $J$-Hermitian translation-invariant kernel 
\begin{equation}\label{kernelAL}
	\begin{split}
		A_L(x_i,y_j)=
		\begin{bmatrix}
			F_L(x_1-y_1)&G_L(x_1-y_2) \\ 
			-\overline{G_L}(y_1-x_2)&H_L(x_2-y_2) 
		\end{bmatrix},
	\end{split}
\end{equation}
where $F_L,G_L,H_L\in L^2(\mathbb{R}^d)$. It follows from Proposition~\ref{prop-2} that for almost every $x\in \mathbb{R}^d$,
\begin{equation*}
	\begin{split}
		|(\mathscr{F}G_L)(x)|^2&\leq \min{\big\{(\mathscr{F}F_L)(x)[1-(\mathscr{F}H_L)(x)],\,(\mathscr{F}H_L)(x)[1-(\mathscr{F}F_L)(x)]\big\}}\\
		&=\frac{1}{2}(\mathscr{F}F_L)(x)+\frac{1}{2}(\mathscr{F}H_L)(x)-(\mathscr{F}F_L)(x)(\mathscr{F}H_L)(x)\\
		&\quad-\frac{1}{2}|(\mathscr{F}F_L)(x)-(\mathscr{F}H_L)(x)|,
	\end{split}
\end{equation*}
and
\begin{equation*}
	\begin{split}
		0\leq |(\mathscr{F}F_L)(x)-(\mathscr{F}H_L)(x)|\leq 1,
	\end{split}
\end{equation*}
hence
\begin{equation*}
	\begin{split}
		\sigma_L^2:&=F_L(0)+H_L(0)-\int_{\mathbb{R}^d}|F_L(x)|^2dx-\int_{\mathbb{R}^d}|H_L(x)|^2dx-2\int_{\mathbb{R}^d}|G_L(x)|^2dx\\
		&=\int_{\mathbb{R}^d}\big[(\mathscr{F}F_L)(x)+(\mathscr{F}H_L)(x)-(\mathscr{F}F_L)(x)^2-(\mathscr{F}H_L)(x)^2-2|(\mathscr{F}G_L)(x)|^2\big]dx\\
		&\geq\int_{\mathbb{R}^d}\big[|(\mathscr{F}F_L)(x)-(\mathscr{F}H_L)(x)|-|(\mathscr{F}F_L)(x)-(\mathscr{F}H_L)(x)|^2\big]dx\\
		&\geq0.
	\end{split}
\end{equation*}
For a real-valued function $f\in L^1(\mathbb{R}^d)\cap L^2(\mathbb{R}^d)$, the linear statistics
\begin{equation*}
	\begin{split}
		S_{(f,-f)_L}(\xi):=\sum\limits_{x\in\xi\cap \mathbb{R}_1^d}f\big (\frac{x}{L}\big )-\sum\limits_{y\in\xi\cap \mathbb{R}_2^d}f\big (\frac{y}{L}\big )\,,\,\,\,\,\xi\in\Conf(\mathbb{R}_1^d\sqcup\mathbb{R}_2^d)=\Conf(\mathbb{R}^d\sqcup\mathbb{R}^d),
	\end{split}
\end{equation*}
as a random variable makes sense. In fact, consider the linear statistics
\begin{equation*}
	\begin{split}
		S_{(|f|,|f|)_L}(\xi)=\sum\limits_{x\in\xi\cap \mathbb{R}_1^d}\big|f\big (\frac{x}{L}\big )\big|+\sum\limits_{y\in\xi\cap \mathbb{R}_2^d}\big|f\big (\frac{y}{L}\big )\big|\,,\,\,\,\,\xi\in\Conf(\mathbb{R}_1^d\sqcup\mathbb{R}_2^d)=\Conf(\mathbb{R}^d\sqcup\mathbb{R}^d),
	\end{split}
\end{equation*}
we can show $\mathbb{E}_LS_{(|f|,|f|)_L}<+\infty$.

\begin{theorem}\label{thm-3}
	If there exist constants $\sigma >0$ and $C$ such that $\sigma_L\to \sigma$ as $L\to +\infty$, $F_L(0)+H_L(0)\leq C$ for all $L\geq 0$. And suppose that there exist $\kappa_L\to +\infty$ as $L\to +\infty$ such that 
	\begin{equation*}
		\begin{split}
			\int_{|x|>\frac{L}{\kappa_L}}\big[|F_L(x)|^2+|H_L(x)|^2+2|G_L(x)|^2\big]dx\to 0\,\,\,\text{as}\,\,\,L\to +\infty.
		\end{split}
    \end{equation*}
	Then for any real-valued function $f\in L^1(\mathbb{R}^d)\cap L^2(\mathbb{R}^d)$, the centered normalized linear statistics
	\begin{equation*}
		\begin{split}
			\frac{1}{\sigma L^{d/2}}\Big\{\big[\sum\limits_{x\in\xi\cap \mathbb{R}_1^d}f\big (\frac{x}{L}\big )-\sum\limits_{y\in\xi\cap \mathbb{R}_2^d}f\big (\frac{y}{L}\big )\big]-\big[F_L(0)-H_L(0)\big]L^d\int_{\mathbb{R}^d}f(x)dx\Big\}
		\end{split}
    \end{equation*}
	converges in distribution to the Gaussian random variable $N(0,\int_{\mathbb{R}^d}f^2(x)dx)$ as $L\to +\infty$.
\end{theorem}

\section{Proof of main results}
\subsection{Proof of Theorem~\ref{thm-1}}
We will use the method of moments to prove Theorem~\ref{thm-1}. Before proof, we remind the reader that for a real-valued random variable $\eta$ with all finite moments, the cumulants $C_n(\eta)$, $n\in\mathbb{N_+}$, are defined through the expansion of the second characteristic function:
\begin{equation}\label{def-cumu}
	\begin{split}
		\log\mathbb{E}(e^{it\eta})=\sum_{n=1}^{N}\frac{C_n(\eta)}{n!}(it)^n+o(|t|^N)\,\,\,\text{as}\,\,\,t\to 0.
	\end{split}
\end{equation}

We will use the following standard lemma, see, e.g. \cite[Theorem~3.3.26]{Du}, \cite[Corollary to Theorem~7.3.3]{Lu} and \cite[Lemma~3]{Soshn}.
\begin{lemma}\label{lemma-1}
	Let $\{\eta_L\}_{L\geq 0}$ be a family of real-valued random variables such that $C_1(\eta_L)=0$, $C_2(\eta_L)=1$ for all $L\geq 0$, and $C_n(\eta_L)\to 0$ as $L\to +\infty$ when $n$ is sufficiently large. Then $\eta_L$ converges in distribution to the standard normal law $N(0,1)$ as $L\to +\infty$.
\end{lemma}

The following lemma is established by Soshnikov in \cite[formulas (2.6) and (2.7)]{Sos}.
\begin{lemma}\label{lemma-2}
	Let $\mathbb{P}$ be a determinantal point process in a locally compact Polish space $X$ with kernel $K$ with respect to Radon measure $\mu$. For $f\in B_c(X)$, the $n$-th $(n\in\mathbb{N_+})$ cumulant of linear statistics $S_f$ can be expressed by
	\begin{equation*}
		\begin{split}
			C_n(S_f)=&\sum_{m=1}^{n}\sum_{\substack{(n_1,\cdots,n_m)\in\mathbb{N}_+^m\\n_1+\cdots+n_m=n}}\frac{(-1)^{m+1}}{m}\frac{n!}{n_1!\cdots n_m!}\\
			&\times\int_{X^m}f^{n_1}(x_1)K(x_1,x_2)f^{n_2}(x_2)K(x_2,x_3)\cdots f^{n_m}(x_m)K(x_m,x_1)d\mu (x_1)\cdots d\mu (x_m).
		\end{split}
	\end{equation*}
\end{lemma}

Using Lemma~\ref{lemma-2} we are able to estimate the cumulants of $S_f$. We are going to prove the following claim.
\begin{claim}\label{claim-3}
	Under the assumptions of Theorem~\ref{thm-1}, for any $m\in\mathbb{N_+}$, any $(n_1,\cdots,n_m)\in\mathbb{N}_+^m$, and any $\varepsilon>0$,
	\begin{equation*}
		\begin{split}
			\Tr(K_Lf_L^{n_1}K_Lf_L^{n_2}\cdots K_Lf_L^{n_m})=O((\Var_L(S_{f_L}))^{\delta+\varepsilon})\,\,\,\text{as}\,\,\,L\to +\infty.
		\end{split}
	\end{equation*}
\end{claim}

Now suppose that we have proved Claim~\ref{claim-3}, then it follows from fomula \eqref{trace-integral} and Lemma~\ref{lemma-2} that for any $n\in\mathbb{N_+}$ and any $\varepsilon>0$,
\begin{equation}\label{estimate-cumu}
	\begin{split}
		C_n(S_{f_L})=O((\Var_L(S_{f_L}))^{\delta+\varepsilon})\,\,\,\text{as}\,\,\,L\to +\infty.
	\end{split}
\end{equation}
Denote
\begin{equation*}
	\begin{split}
		\eta_L:=\frac{S_{f_L}-\mathbb{E}_LS_{f_L}}{\sqrt{{\Var_L}S_{f_L}}}.
	\end{split}
\end{equation*}
Then by the definition of cumulants \eqref{def-cumu}, it is easy to see that $C_1(\eta_L)=0$, $C_2(\eta_L)=1$ for all $L\geq 0$. And when $n>2$, for any $\varepsilon>0$, by \eqref{estimate-cumu} we obtain 
\begin{equation*}
	C_n(\eta_L)=\frac{C_n(S_{f_L})}{(\Var_L(S_{f_L}))^{\frac{n}{2}}}=\frac{O((\Var_L(S_{f_L}))^{\delta+\varepsilon})}{(\Var_L(S_{f_L}))^{\frac{n}{2}}}\,\,\,\text{as}\,\,\,L\to +\infty\,,
\end{equation*}
hence $C_n(\eta_L)\to 0$ as $L\to +\infty$ when $n>\max{\{2,2\delta\}}$. Thus by Lemma~\ref{lemma-1}, $\eta_L$ converges in distribution to $N(0,1)$ as $L\to +\infty$.

\begin{proof}[Proof of Claim~\ref{claim-3}]
Note that $\Tr(K_Lf_L^{n_1}K_Lf_L^{n_2}\cdots K_Lf_L^{n_m})$ is a linear combination of terms $\Tr(K_Lf_{L,\upsilon_1}^{n_1}K_Lf_{L,\upsilon_2}^{n_2}\cdots K_Lf_{L,\upsilon_m}^{n_m})$, where $\upsilon_i$ represents $+$ or $-$, $i=1,2,\cdots,m$, and $f_{L,+}=\max{\{f_L,0\}}$, $f_{L,-}=\max{\{-f_L,0\}}$. Let us fix the choice of $\upsilon_i$ in each of the factors, it is enough to prove that for any $\varepsilon>0$,
\begin{equation*}
	\begin{split}
		\Tr(K_Lf_{L,\upsilon_1}^{n_1}K_Lf_{L,\upsilon_2}^{n_2}\cdots K_Lf_{L,\upsilon_m}^{n_m})=O((\Var_L(S_{f_L}))^{\delta+\varepsilon})\,\,\,\text{as}\,\,\,L\to +\infty.
	\end{split}
\end{equation*}

For the case $m=1$, by \eqref{def-dpp} and \eqref{trace-integral}, we have
\begin{equation}\label{Tr(Kf)}
	\begin{split}
		\Tr(K_Lf_{L,\upsilon_1}^{n_1})=\int_{X_L}f_{L,\upsilon_1}^{n_1}(x)K_L(x,x)d\mu_L(x)\geq0.
	\end{split}
\end{equation}
It follows from \eqref{expecsf} that
\begin{equation*}
   	\begin{split}
		|\Tr(K_Lf_{L,\upsilon_1}^{n_1})|&=\Tr(K_Lf_{L,\upsilon_1}^{n_1})\\
     	&\leq\left\|f_{L,\upsilon_1}\right\|_\infty^{n_1-1}\int_{X_L}f_{L,\upsilon_1}(x)K_L(x,x)d\mu_L(x)\\
   		&\leq\left\|f_L\right\|_\infty^{n_1-1}\int_{X_L}|f_L(x)|K_L(x,x)d\mu_L(x)\\
   		&=\left\|f_L\right\|_\infty^{n_1-1}\mathbb{E}_LS_{|f_L|}\\
   		&=O((\Var_L(S_{f_L}))^{\delta+\varepsilon}).
   	\end{split}
\end{equation*}

For the case $m=2$, 
\begin{equation*}
	\begin{split}
		\Tr(K_Lf_{L,\upsilon_1}^{n_1}K_Lf_{L,\upsilon_2}^{n_2})&=\Tr\big[K_L(P_L^{(1)}f_{L,\upsilon_1}^{n_1}+P_L^{(2)}f_{L,\upsilon_1}^{n_1})K_L(P_L^{(1)}f_{L,\upsilon_2}^{n_2}+P_L^{(2)}f_{L,\upsilon_2}^{n_2})\big]\\
		&=\Tr(K_LP_L^{(1)}f_{L,\upsilon_1}^{n_1}K_LP_L^{(1)}f_{L,\upsilon_2}^{n_2})+\Tr(K_LP_L^{(1)}f_{L,\upsilon_1}^{n_1}K_LP_L^{(2)}f_{L,\upsilon_2}^{n_2})\\
		&\quad+\Tr(K_LP_L^{(2)}f_{L,\upsilon_1}^{n_1}K_LP_L^{(2)}f_{L,\upsilon_2}^{n_2})+\Tr(K_LP_L^{(2)}f_{L,\upsilon_1}^{n_1}K_LP_L^{(1)}f_{L,\upsilon_2}^{n_2})\\
		&=I_1(L)+I_2(L)+I_3(L)+I_4(L),
    \end{split}
\end{equation*}
where $P_L^{(i)}$ is the orthogonal projection of $L^2(X_L,\mu_L)$ onto $L^2(X_L^{(i)},\mu_L)$, $i=1,2$. Since the arguments for $I_3(L)$ and $I_4(L)$ are similar to those for $I_1(L)$ and $I_2(L)$, we only prove $I_1(L)=O((\Var_L(S_{f_L}))^{\delta+\varepsilon})$ and $I_2(L)=O((\Var_L(S_{f_L}))^{\delta+\varepsilon})$.

Note that from \eqref{J-Hermitian}, we have $(P_L^{(1)}K_LP_L^{(1)})^*=P_L^{(1)}K_LP_L^{(1)}$, hence
\begin{equation}\label{I1(L)}
	\begin{split}
		I_1(L)&=\Tr(K_LP_L^{(1)}f_{L,\upsilon_1}^{n_1}K_LP_L^{(1)}f_{L,\upsilon_2}^{n_2})\\
		&=\Tr(K_LP_L^{(1)}f_{L,\upsilon_1}^{n_1}P_L^{(1)}K_LP_L^{(1)}f_{L,\upsilon_2}^{n_2}P_L^{(1)})\\
		&=\Tr(f_{L,\upsilon_2}^{n_2/2}P_L^{(1)}K_LP_L^{(1)}f_{L,\upsilon_1}^{n_1}P_L^{(1)}K_LP_L^{(1)}f_{L,\upsilon_2}^{n_2/2})\\
		&=\Tr\big[(f_{L,\upsilon_2}^{n_2/2}P_L^{(1)}K_LP_L^{(1)}f_{L,\upsilon_1}^{n_1/2})(f_{L,\upsilon_2}^{n_2/2}P_L^{(1)}K_LP_L^{(1)}f_{L,\upsilon_1}^{n_1/2})^*\big]\\
		&\geq0.
	\end{split}
\end{equation}
It follows from \eqref{trace-integral} and \eqref{varsf} that 
\begin{equation*}
	\begin{split}
		\Var_L(S_{P_L^{(1)}f_{L,\upsilon_1}^{n_1}+P_L^{(1)}f_{L,\upsilon_2}^{n_2}})&=\Tr\big[K_L(P_L^{(1)}f_{L,\upsilon_1}^{n_1}+P_L^{(1)}f_{L,\upsilon_2}^{n_2})^2\big]\\
		&\quad-\Tr\big[K_L(P_L^{(1)}f_{L,\upsilon_1}^{n_1}+P_L^{(1)}f_{L,\upsilon_2}^{n_2})K_L(P_L^{(1)}f_{L,\upsilon_1}^{n_1}+P_L^{(1)}f_{L,\upsilon_2}^{n_2})\big]\\
		&=\Tr(K_LP_L^{(1)}f_{L,\upsilon_1}^{2n_1})+\Tr(K_LP_L^{(1)}f_{L,\upsilon_2}^{2n_2})+2\Tr(K_LP_L^{(1)}f_{L,\upsilon_1}^{n_1}f_{L,\upsilon_2}^{n_2})\\
		&\quad-\Tr(K_LP_L^{(1)}f_{L,\upsilon_1}^{n_1}K_LP_L^{(1)}f_{L,\upsilon_1}^{n_1})-\Tr(K_LP_L^{(1)}f_{L,\upsilon_2}^{n_2}K_LP_L^{(1)}f_{L,\upsilon_2}^{n_2})\\
		&\quad-2I_1(L)\\
		&\leq\Tr(K_Lf_{L,\upsilon_1}^{2n_1})+\Tr(K_Lf_{L,\upsilon_2}^{2n_2})+2\Tr(K_Lf_{L,\upsilon_1}^{n_1+n_2})-2I_1(L),
	\end{split}
\end{equation*}
where the last inequality is based on \eqref{Tr(Kf)} and \eqref{I1(L)}. Note that the variance of a random variable is non-negative, we get
\begin{equation*}
	\begin{split}
		|I_1(L)|&=I_1(L)\\
		&\leq\frac{1}{2}\big[\Tr(K_Lf_{L,\upsilon_1}^{2n_1})+\Tr(K_Lf_{L,\upsilon_2}^{2n_2})+2\Tr(K_Lf_{L,\upsilon_1}^{n_1+n_2})\big]\\
		&=O((\Var_L(S_{f_L}))^{\delta+\varepsilon}),
	\end{split}
\end{equation*}
where the last equality is from the case $m=1$.

As for $I_2(L)$, note that from \eqref{J-Hermitian}, we have $(P_L^{(2)}K_LP_L^{(1)})^*=-P_L^{(1)}K_LP_L^{(2)}$, hence
\begin{equation}\label{I2(L)}
	\begin{split}
		I_2(L)&=\Tr(K_LP_L^{(1)}f_{L,\upsilon_1}^{n_1}K_LP_L^{(2)}f_{L,\upsilon_2}^{n_2})\\
		&=\Tr(K_LP_L^{(1)}f_{L,\upsilon_1}^{n_1}P_L^{(1)}K_LP_L^{(2)}f_{L,\upsilon_2}^{n_2}P_L^{(2)})\\
		&=\Tr(f_{L,\upsilon_2}^{n_2/2}P_L^{(2)}K_LP_L^{(1)}f_{L,\upsilon_1}^{n_1}P_L^{(1)}K_LP_L^{(2)}f_{L,\upsilon_2}^{n_2/2})\\
		&=-\Tr\big[(f_{L,\upsilon_2}^{n_2/2}P_L^{(2)}K_LP_L^{(1)}f_{L,\upsilon_1}^{n_1/2})(f_{L,\upsilon_2}^{n_2/2}P_L^{(2)}K_LP_L^{(1)}f_{L,\upsilon_1}^{n_1/2})^*\big]\\
		&\leq0.
	\end{split}
\end{equation}
It follows from \eqref{trace-integral} and \eqref{varsf} that 
\begin{equation*}
	\begin{split}
		\Var_L(S_{P_L^{(1)}f_{L,\upsilon_1}^{n_1}+P_L^{(2)}f_{L,\upsilon_2}^{n_2}})&=\Tr\big[K_L(P_L^{(1)}f_{L,\upsilon_1}^{n_1}+P_L^{(2)}f_{L,\upsilon_2}^{n_2})^2\big]\\
		&\quad-\Tr\big[K_L(P_L^{(1)}f_{L,\upsilon_1}^{n_1}+P_L^{(2)}f_{L,\upsilon_2}^{n_2})K_L(P_L^{(1)}f_{L,\upsilon_1}^{n_1}+P_L^{(2)}f_{L,\upsilon_2}^{n_2})\big]\\
		&=\Tr(K_LP_L^{(1)}f_{L,\upsilon_1}^{2n_1})+\Tr(K_LP_L^{(2)}f_{L,\upsilon_2}^{2n_2})\\
		&\quad-\Tr(K_LP_L^{(1)}f_{L,\upsilon_1}^{n_1}K_LP_L^{(1)}f_{L,\upsilon_1}^{n_1})-\Tr(K_LP_L^{(2)}f_{L,\upsilon_2}^{n_2}K_LP_L^{(2)}f_{L,\upsilon_2}^{n_2})\\
		&\quad-2I_2(L),
	\end{split}
\end{equation*}
and
\begin{equation*}
	\begin{split}
		\Var_L(S_{P_L^{(1)}f_{L,\upsilon_1}^{n_1}})+\Var_L(S_{P_L^{(2)}f_{L,\upsilon_2}^{n_2}})&=\Tr(K_LP_L^{(1)}f_{L,\upsilon_1}^{2n_1})-\Tr(K_LP_L^{(1)}f_{L,\upsilon_1}^{n_1}K_LP_L^{(1)}f_{L,\upsilon_1}^{n_1})\\
		&\quad+\Tr(K_LP_L^{(2)}f_{L,\upsilon_2}^{2n_2})-\Tr(K_LP_L^{(2)}f_{L,\upsilon_2}^{n_2}K_LP_L^{(2)}f_{L,\upsilon_2}^{n_2}).
	\end{split}
\end{equation*}
Using the basic fact
\begin{equation*}
	\begin{split}
		\Var_L(S_{P_L^{(1)}f_{L,\upsilon_1}^{n_1}+P_L^{(2)}f_{L,\upsilon_2}^{n_2}})&=\Var_L(S_{P_L^{(1)}f_{L,\upsilon_1}^{n_1}}+S_{P_L^{(2)}f_{L,\upsilon_2}^{n_2}})\\
		&\leq2\big[\Var_L(S_{P_L^{(1)}f_{L,\upsilon_1}^{n_1}})+\Var_L(S_{P_L^{(2)}f_{L,\upsilon_2}^{n_2}})\big],
	\end{split}
\end{equation*}
together with \eqref{Tr(Kf)}, \eqref{I1(L)}, \eqref{I2(L)} and the case $m=1$, we obtain
\begin{equation*}
	\begin{split}
		|I_2(L)|&=-I_2(L)\\
		&\leq\frac{1}{2}\big[\Tr(K_LP_L^{(1)}f_{L,\upsilon_1}^{2n_1})+\Tr(K_LP_L^{(2)}f_{L,\upsilon_2}^{2n_2})\\
		&\quad\quad-\Tr(K_LP_L^{(1)}f_{L,\upsilon_1}^{n_1}K_LP_L^{(1)}f_{L,\upsilon_1}^{n_1})-\Tr(K_LP_L^{(2)}f_{L,\upsilon_2}^{n_2}K_LP_L^{(2)}f_{L,\upsilon_2}^{n_2})\big]\\
		&\leq\frac{1}{2}\big[\Tr(K_Lf_{L,\upsilon_1}^{2n_1})+\Tr(K_Lf_{L,\upsilon_2}^{2n_2})\big]\\
		&=O((\Var_L(S_{f_L}))^{\delta+\varepsilon}),
	\end{split}
\end{equation*}
where we used the inequality $\Tr(K_LP_L^{(2)}f_{L,\upsilon_2}^{n_2}K_LP_L^{(2)}f_{L,\upsilon_2}^{n_2})\geq0$ by arguments similar to \eqref{I1(L)}.

For the case $m\geq3$, using the Cauchy-Schwarz inequality $|\Tr(AB)|^2\leq\Tr(AA^*)\Tr(BB^*)$ for Hilbert-Schmidt operators $A$ and $B$, we have
\begin{equation*}
	\begin{split}
		&\quad\,\,\Big|\Tr(K_Lf_{L,\upsilon_1}^{n_1}K_Lf_{L,\upsilon_2}^{n_2}K_Lf_{L,\upsilon_3}^{n_3}K_Lf_{L,\upsilon_4}^{n_4}\cdots K_Lf_{L,\upsilon_m}^{n_m})\Big|^2\\
		&=\Big|\Tr\big[(f_{L,\upsilon_1}^{n_1}K_Lf_{L,\upsilon_2}^{{n_2}/2})(f_{L,\upsilon_2}^{{n_2}/2}K_Lf_{L,\upsilon_3}^{n_3}K_Lf_{L,\upsilon_4}^{n_4}\cdots K_Lf_{L,\upsilon_m}^{n_m}K_L)\big]\Big|^2\\
		&\leq\Tr\big[(f_{L,\upsilon_1}^{n_1}K_Lf_{L,\upsilon_2}^{{n_2}/2})(f_{L,\upsilon_1}^{n_1}K_Lf_{L,\upsilon_2}^{{n_2}/2})^*\big]\\
		&\quad\times\Tr\big[(f_{L,\upsilon_2}^{{n_2}/2}K_Lf_{L,\upsilon_3}^{n_3}K_Lf_{L,\upsilon_4}^{n_4}\cdots K_Lf_{L,\upsilon_m}^{n_m}K_L)(f_{L,\upsilon_2}^{{n_2}/2}K_Lf_{L,\upsilon_3}^{n_3}K_Lf_{L,\upsilon_4}^{n_4}\cdots K_Lf_{L,\upsilon_m}^{n_m}K_L)^*\big]\\
		&=\Tr(K_Lf_{L,\upsilon_2}^{{n_2}}K_L^*f_{L,\upsilon_1}^{{2n_1}})\\
		&\quad\times\Tr(f_{L,\upsilon_2}^{{n_2}/2}K_Lf_{L,\upsilon_3}^{n_3}K_Lf_{L,\upsilon_4}^{n_4}\cdots K_Lf_{L,\upsilon_m}^{n_m}K_LK_L^*f_{L,\upsilon_m}^{n_m}K_L^*\cdots f_{L,\upsilon_4}^{n_4}K_L^*f_{L,\upsilon_3}^{n_3}K_L^*f_{L,\upsilon_2}^{{n_2}/2}).\\
	\end{split}
\end{equation*}
And using the inequality $|\Tr(AB)|\leq\Tr(A)\left\|B\right\|$ for positive operator $A$ which is trace-class and bounded linear operator $B$, see, e.g. \cite[Theorem~2.7, ~2.14]{Si}, we have
\begin{equation*}
	\begin{split}
		&\quad\,\,\Tr(f_{L,\upsilon_2}^{{n_2}/2}K_Lf_{L,\upsilon_3}^{n_3}K_Lf_{L,\upsilon_4}^{n_4}\cdots K_Lf_{L,\upsilon_m}^{n_m}K_LK_L^*f_{L,\upsilon_m}^{n_m}K_L^*\cdots f_{L,\upsilon_4}^{n_4}K_L^*f_{L,\upsilon_3}^{n_3}K_L^*f_{L,\upsilon_2}^{{n_2}/2})\\
		&=\Tr(f_{L,\upsilon_3}^{n_3}K_L^*f_{L,\upsilon_2}^{{n_2}/2}f_{L,\upsilon_2}^{{n_2}/2}K_Lf_{L,\upsilon_3}^{n_3}K_Lf_{L,\upsilon_4}^{n_4}\cdots K_Lf_{L,\upsilon_m}^{n_m}K_LK_L^*f_{L,\upsilon_m}^{n_m}K_L^*\cdots f_{L,\upsilon_4}^{n_4}K_L^*)\\
		&=\Tr\big[(f_{L,\upsilon_2}^{{n_2}/2}K_Lf_{L,\upsilon_3}^{n_3})^*(f_{L,\upsilon_2}^{{n_2}/2}K_Lf_{L,\upsilon_3}^{n_3})(K_Lf_{L,\upsilon_4}^{n_4}\cdots K_Lf_{L,\upsilon_m}^{n_m}K_L)(K_Lf_{L,\upsilon_4}^{n_4}\cdots K_Lf_{L,\upsilon_m}^{n_m}K_L)^*\big]\\
		&\leq\Tr\big[(f_{L,\upsilon_2}^{{n_2}/2}K_Lf_{L,\upsilon_3}^{n_3})^*(f_{L,\upsilon_2}^{{n_2}/2}K_Lf_{L,\upsilon_3}^{n_3})\big]\left\|(K_Lf_{L,\upsilon_4}^{n_4}\cdots K_Lf_{L,\upsilon_m}^{n_m}K_L)(K_Lf_{L,\upsilon_4}^{n_4}\cdots K_Lf_{L,\upsilon_m}^{n_m}K_L)^*\right\|\\
		&\leq\Tr(K_Lf_{L,\upsilon_3}^{{2n_3}}K_L^*f_{L,\upsilon_2}^{{n_2}})\left\|K_Lf_{L,\upsilon_4}^{n_4}\cdots K_Lf_{L,\upsilon_m}^{n_m}K_L\right\|^2.
	\end{split}
\end{equation*}
These estimations deduce that
\begin{equation*}
	\begin{split}
		&\quad\,\,\Big|\Tr(K_Lf_{L,\upsilon_1}^{n_1}K_Lf_{L,\upsilon_2}^{n_2}K_Lf_{L,\upsilon_3}^{n_3}K_Lf_{L,\upsilon_4}^{n_4}\cdots K_Lf_{L,\upsilon_m}^{n_m})\Big|^2\\
		&\leq\Tr(K_Lf_{L,\upsilon_2}^{{n_2}}K_L^*f_{L,\upsilon_1}^{{2n_1}})\Tr(K_Lf_{L,\upsilon_3}^{{2n_3}}K_L^*f_{L,\upsilon_2}^{{n_2}})\left\|K_Lf_{L,\upsilon_4}^{n_4}\cdots K_Lf_{L,\upsilon_m}^{n_m}K_L\right\|^2\\
		&=J_1(L)J_2(L)\left\|D_L\right\|^2.
	\end{split}
\end{equation*}
Based on \eqref{0<Khat<1} and \cite[Proposition~7]{Ly}, we have $\left\|K_L\right\|\leq1$, then
\begin{equation*}
	\begin{split}
		\left\|D_L\right\|&=\left\|K_Lf_{L,\upsilon_4}^{n_4}\cdots K_Lf_{L,\upsilon_m}^{n_m}K_L\right\|\\
		&\leq\left\|f_{L,\upsilon_4}\right\|_{\infty}^{n_4}\cdots\left\|f_{L,\upsilon_m}\right\|_{\infty}^{n_m}\left\|K_L\right\|^{m-2}\\
		&\leq\left\|f_L\right\|_{\infty}^{n_4+\cdots+n_m}\\
		&=o((\Var_L(S_{f_L}))^\varepsilon).
	\end{split}
\end{equation*}
Split $J_1(L)$ into four parts as follow:
\begin{equation*}
	\begin{split}
		J_1(L)&=\Tr(K_Lf_{L,\upsilon_2}^{{n_2}}K_L^*f_{L,\upsilon_1}^{{2n_1}})\\
		&=\Tr\big[K_L(P_L^{(1)}f_{L,\upsilon_2}^{{n_2}}+P_L^{(2)}f_{L,\upsilon_2}^{{n_2}})K_L^*(P_L^{(1)}f_{L,\upsilon_1}^{{2n_1}}+P_L^{(2)}f_{L,\upsilon_1}^{{2n_1}})\big]\\
		&=\Tr(K_LP_L^{(1)}f_{L,\upsilon_2}^{{n_2}}K_L^*P_L^{(1)}f_{L,\upsilon_1}^{{2n_1}})+\Tr(K_LP_L^{(1)}f_{L,\upsilon_2}^{{n_2}}K_L^*P_L^{(2)}f_{L,\upsilon_1}^{{2n_1}})\\
		&\quad+\Tr(K_LP_L^{(2)}f_{L,\upsilon_2}^{{n_2}}K_L^*P_L^{(1)}f_{L,\upsilon_1}^{{2n_1}})+\Tr(K_LP_L^{(2)}f_{L,\upsilon_2}^{{n_2}}K_L^*P_L^{(2)}f_{L,\upsilon_1}^{{2n_1}}).\\
	\end{split}
\end{equation*}
It follows from \eqref{J-Hermitian} and the case $m=2$ that
\begin{equation*}
	\begin{split}
		\Tr(K_LP_L^{(1)}f_{L,\upsilon_2}^{{n_2}}K_L^*P_L^{(1)}f_{L,\upsilon_1}^{{2n_1}})&=\Tr\big[K_Lf_{L,\upsilon_2}^{{n_2}}(P_L^{(1)}K_L^*P_L^{(1)})f_{L,\upsilon_1}^{{2n_1}}\big]\\
		&=\Tr\big[K_Lf_{L,\upsilon_2}^{{n_2}}(P_L^{(1)}K_LP_L^{(1)})f_{L,\upsilon_1}^{{2n_1}}\big]\\
		&=\Tr(K_LP_L^{(1)}f_{L,\upsilon_2}^{{n_2}}K_LP_L^{(1)}f_{L,\upsilon_1}^{{2n_1}})\\
		&=O((\Var_L(S_{f_L}))^{\delta+\varepsilon}),
	\end{split}
\end{equation*}
and
\begin{equation*}
	\begin{split}
	    \Tr(K_LP_L^{(1)}f_{L,\upsilon_2}^{{n_2}}K_L^*P_L^{(2)}f_{L,\upsilon_1}^{{2n_1}})&=\Tr\big[K_Lf_{L,\upsilon_2}^{{n_2}}(P_L^{(1)}K_L^*P_L^{(2)})f_{L,\upsilon_1}^{{2n_1}}\big]\\
		&=-\Tr\big[K_Lf_{L,\upsilon_2}^{{n_2}}(P_L^{(1)}K_LP_L^{(2)})f_{L,\upsilon_1}^{{2n_1}}\big]\\
		&=-\Tr(K_LP_L^{(1)}f_{L,\upsilon_2}^{{n_2}}K_LP_L^{(2)}f_{L,\upsilon_1}^{{2n_1}})\\
		&=O((\Var_L(S_{f_L}))^{\delta+\varepsilon}). 
	\end{split}
\end{equation*}
Similarly, we have
\begin{equation*}
        \Tr(K_LP_L^{(2)}f_{L,\upsilon_2}^{{n_2}}K_L^*P_L^{(1)}f_{L,\upsilon_1}^{{2n_1}})=O((\Var_L(S_{f_L}))^{\delta+\varepsilon}),
\end{equation*}
and
\begin{equation*}
        \Tr(K_LP_L^{(2)}f_{L,\upsilon_2}^{{n_2}}K_L^*P_L^{(2)}f_{L,\upsilon_1}^{{2n_1}})=O((\Var_L(S_{f_L}))^{\delta+\varepsilon}), 
\end{equation*}
hence we obtain $J_1(L)=O((\Var_L(S_{f_L}))^{\delta+\varepsilon})$.
As for $J_2(L)$, we also have $J_2(L)=O((\Var_L(S_{f_L}))^{\delta+\varepsilon})$ by arguments similar to $J_1(L)$. Together with the estimations of $\left\|D_L\right\|$, $J_1(L)$ and $J_2(L)$, this concludes that
\begin{equation*}
	\begin{split}
		\Tr(K_Lf_{L,\upsilon_1}^{n_1}K_Lf_{L,\upsilon_2}^{n_2}K_Lf_{L,\upsilon_3}^{n_3}K_Lf_{L,\upsilon_4}^{n_4}\cdots K_Lf_{L,\upsilon_m}^{n_m})=O((\Var_L(S_{f_L}))^{\delta+\varepsilon}).
	\end{split}
\end{equation*}

This completes the proof of Theorem~\ref{thm-1}.
\end{proof}

\subsection{Proof of Proposition~\ref{prop-2}}
By slightly abusing the notation, the associated integral operator $A$ with $J$-Hermitian translation-invariant kernel $A$ in \eqref{J-HT-IIK} can be written in block form:
\begin{equation*}
	\begin{split}
		A=
		\begin{bmatrix}
			F&G\\
			-G^*&H
		\end{bmatrix},
	\end{split}
\end{equation*}
and then the operator $\widehat A$ denoted by \eqref{def-Khat} has the following block form:
\begin{equation*}
	\begin{split}
		\widehat A=
		\begin{bmatrix}
			F&G\\
			G^*&P_2-H
		\end{bmatrix}.
	\end{split}
\end{equation*}
It follows from \eqref{0<Khat<1} that there exists a determinantal point process on $\mathbb{R}^d\sqcup \mathbb{R}^d$ with kernel $A$ if and only if $0\leq\widehat A\leq1$.

Using the fomula $\mathscr{F}(T_1*T_2)=(\mathscr{F}T_1)(\mathscr{F}T_2)$ when $T_1, T_2\in L^2(\mathbb{R}^d)$, for any function in $L^2(\mathbb{R}^d\sqcup\mathbb{R}^d)=L^2(\mathbb{R}^d)\oplus L^2(\mathbb{R}^d)$, i.e. for any $f_1,f_2\in L^2(\mathbb{R}^d)$, we have 
\begin{equation*}
	\begin{split}
		\begin{bmatrix}
			\mathscr{F}&0\\
			0&\mathscr{F}
		\end{bmatrix}
	    \begin{bmatrix}
	    	F&G\\
	    	G^*&P_2-H
	    \end{bmatrix}
        \begin{bmatrix}
        	f_1\\
        	f_2
        \end{bmatrix}
        =
        \begin{bmatrix}
        	\mathscr{F}F&\mathscr{F}G\\
        	\overline{\mathscr{F}G}&1-\mathscr{F}H
        \end{bmatrix}
        \begin{bmatrix}
        	\mathscr{F}&0\\
        	0&\mathscr{F}
        \end{bmatrix}
        \begin{bmatrix}
        	f_1\\
        	f_2
        \end{bmatrix},
	\end{split}
\end{equation*}
hence
\begin{equation*}
	\begin{split}
		\begin{bmatrix}
			F&G\\
			G^*&P_2-H
		\end{bmatrix}
	    =
	    \begin{bmatrix}
	    	\mathscr{F}^{-1}&0\\
	    	0&\mathscr{F}^{-1}
	    \end{bmatrix}
		\begin{bmatrix}
			\mathscr{F}F&\mathscr{F}G\\
			\overline{\mathscr{F}G}&1-\mathscr{F}H
		\end{bmatrix}
		\begin{bmatrix}
			\mathscr{F}&0\\
			0&\mathscr{F}
		\end{bmatrix}.
	\end{split}
\end{equation*}
Since the Fourier transform $\begin{bmatrix}\mathscr{F}&0\\0&\mathscr{F}\end{bmatrix}$ is an unitary operator on $L^2(\mathbb{R}^d\sqcup\mathbb{R}^d)$, the operator $\widehat A=\begin{bmatrix}F&G\\G^*&P_2-H\end{bmatrix}$ and the multiplication operator $\begin{bmatrix}\mathscr{F}F&\mathscr{F}G\\\overline{\mathscr{F}G}&1-\mathscr{F}H
\end{bmatrix}$ are unitary equivalent, hence $0\leq\widehat A\leq1$ is equivalent to
\begin{equation}\label{Four-tran-matr}
	\begin{split}
		0\leq
		\begin{bmatrix}
			\mathscr{F}F&\mathscr{F}G\\
			\overline{\mathscr{F}G}&1-\mathscr{F}H
		\end{bmatrix}
		\leq1.
	\end{split}
\end{equation}

Denote
\begin{equation*}
	\begin{split}
		M_1(x)=\begin{bmatrix}(\mathscr{F}F)(x)&(\mathscr{F}G)(x)\\\overline{(\mathscr{F}G)}(x)&1-(\mathscr{F}H)(x)\end{bmatrix} \text{and } M_2(x)=\begin{bmatrix}1-(\mathscr{F}F)(x)&-(\mathscr{F}G)(x)\\-\overline{(\mathscr{F}G)}(x)&(\mathscr{F}H)(x)\end{bmatrix}.
	\end{split}
\end{equation*}
Hence \eqref{Four-tran-matr} implies that Proposition~\ref{prop-2} requires us to find the necessary and sufficient conditions such that
\begin{equation}\label{integral1}
	\begin{split}
		\int_{\mathbb{R}^d}\big(\overline{f_1}(x)\,\,\,\,\overline{f_2}(x)\big)M_1(x)\big({f_1}(x)\,\,\,\,{f_2}(x)\big)^{\mathrm{T}}dx\geq0
	\end{split}
\end{equation}
and
\begin{equation}\label{integral2}
	\begin{split}
		\int_{\mathbb{R}^d}\big(\overline{f_1}(x)\,\,\,\,\overline{f_2}(x)\big)M_2(x)\big({f_1}(x)\,\,\,\,{f_2}(x)\big)^{\mathrm{T}}dx\geq0
	\end{split}
\end{equation}
for any $f_1,f_2\in L^2(\mathbb{R}^d)$.
\begin{claim}\label{claim-4}
	\eqref{integral1} and \eqref{integral2} hold for any $f_1,f_2\in L^2(\mathbb{R}^d)$ if and only if for almost every $x\in\mathbb{R}^d$, the matrices $M_1(x)$ and $M_2(x)$ are non-negative definite.
\end{claim}

For a fixed $x\in\mathbb{R}^d$, the matrix $M_1(x)$ is non-negative definite is equivalent to saying that $(\mathscr{F}F)(x)\geq0$, $1-(\mathscr{F}H)(x)\geq0$, and $\det[M_1(x)]\geq0$, that is 
\begin{equation}\label{Fourier1}
	\begin{split}
		\left\{\begin{array}{l}
			(\mathscr{F}F)(x)\geq0\\
			(\mathscr{F}H)(x)\leq1\\
			|(\mathscr{F}G)(x)|^2\leq (\mathscr{F}F)(x)[1-(\mathscr{F}H)(x)]
		\end{array}\right..
	\end{split}
\end{equation}
Similarly, the matrix $M_2(x)$ is non-negative definite is equivalent to
\begin{equation}\label{Fourier2}
	\begin{split}
		\left\{\begin{array}{l}
			(\mathscr{F}F)(x)\leq1\\
			(\mathscr{F}H)(x)\geq0\\
			|(\mathscr{F}G)(x)|^2\leq (\mathscr{F}H)(x)[1-(\mathscr{F}F)(x)]
		\end{array}\right..
	\end{split}
\end{equation}
Combining \eqref{Fourier1}, \eqref{Fourier2} with Claim \ref{claim-4}, Prpposition \ref{prop-2} follows.

\begin{proof}[Proof of Claim~\ref{claim-4}]
If for almost every $x\in\mathbb{R}^d$, the matrices $M_1(x)$ and $M_2(x)$ are non-negative definite, then for any $f_1,f_2\in L^2(\mathbb{R}^d)$, we have
\begin{equation*}
	\begin{split}
		\big(\overline{f_1}(x)\,\,\,\,\overline{f_2}(x)\big)M_1(x)\big({f_1}(x)\,\,\,\,{f_2}(x)\big)^{\mathrm{T}}\geq0\,\,\,\,\text{a.e.}\,\,x\in\mathbb{R}^d,
	\end{split}
\end{equation*}
and
\begin{equation*}
	\begin{split}
		\big(\overline{f_1}(x)\,\,\,\,\overline{f_2}(x)\big)M_2(x)\big({f_1}(x)\,\,\,\,{f_2}(x)\big)^{\mathrm{T}}\geq0\,\,\,\,\text{a.e.}\,\,x\in\mathbb{R}^d,
	\end{split}
\end{equation*}
these imply \eqref{integral1} and \eqref{integral2}.

Conversely, we shall prove that for almost every $x\in\mathbb{R}^d$, the matrix $M_1(x)$ is non-negative definite and we omit the arguments of $M_2(x)$ which is similar to $M_1(x)$. Fix any $N\in\mathbb{N_+}$, it is enough to prove that the matrix $M_1(x)$ is non-negative definite for almost every $x\in[-N,N]^d$. Let $\mathcal{Q}=\{q_n\}_{n=1}^{\infty}$ be a fixed countable dense subset of $\mathbb{C}$. For $m,n\in\mathbb{N_+}$, set $f_1=q_m\chi_{[-N,N]^d}$, $f_2=q_n\chi_{[-N,N]^d}$, then we have
\begin{equation*}
	\begin{split}
		\int_{[-N,N]^d}(\overline{q_m}\,\,\,\,\overline{q_n})M_1(x)(q_m\,\,\,q_n)^{\mathrm{T}}dx\geq0.
	\end{split}
\end{equation*}
Since almost every $x\in[-N,N]^d$ is Lebesgue point of the integrand $(\overline{q_m}\,\,\,\,\overline{q_n})M_1(x)(q_m\,\,\,q_n)^{\mathrm{T}}$, and the integral over any ball supported in ${[-N,N]^d}$ is always non-negative by \eqref{integral1}, there is a zero measure set $E_{mn}$ such that for any $x\in[-N,N]^d\big\backslash E_{mn}$, 
\begin{equation*}
	\begin{split}
		(\overline{q_m}\,\,\,\,\overline{q_n})M_1(x)(q_m\,\,\,q_n)^{\mathrm{T}}\geq0.
	\end{split}
\end{equation*}
Hence for any $x\in[-N,N]^d\big\backslash\big(\bigcup_{m,n\geq1}E_{mn}\big)$ and any $p,q\in\mathcal{Q}$, 
\begin{equation*}
	\begin{split}
		(\overline{p}\,\,\,\overline{q})M_1(x)(p\,\,\,q)^{\mathrm{T}}\geq0.
	\end{split}
\end{equation*}
This yields that for any $x\in[-N,N]^d\big\backslash\big(\bigcup_{m,n\geq1}E_{mn}\big)$ and any $a,b\in\mathbb{C}$,
\begin{equation*}
	\begin{split}
		(\overline{a}\,\,\,\overline{b})M_1(x)(a\,\,\,b)^{\mathrm{T}}\geq0.
	\end{split}
\end{equation*}
Note that $\bigcup _{m,n\geq1}E_{mn}$ has zero measure, thus for almost every $x\in[-N,N]^d$, the matrix $M_1(x)$ is non-negative definite. This shows that for almost every $x\in\mathbb{R}^d$, the matrix $M_1(x)$ is non-negative definite.

This completes the proof of Proposition~\ref{prop-2}.
\end{proof}

\subsection{Proof of Theorem~\ref{thm-3}}
For a real-valued function $f\in L^1(\mathbb{R}^d)\cap L^2(\mathbb{R}^d)$ and $L\geq0$, let us first consider the linear statistics
\begin{equation*}
	\begin{split}
		S_{(|f|,|f|)_L}(\xi)=\sum\limits_{x\in\xi\cap \mathbb{R}_1^d}\big|f\big (\frac{x}{L}\big )\big|+\sum\limits_{y\in\xi\cap \mathbb{R}_2^d}\big|f\big (\frac{y}{L}\big )\big|\,,\,\,\,\,\xi\in\Conf(\mathbb{R}_1^d\sqcup\mathbb{R}_2^d)=\Conf(\mathbb{R}^d\sqcup\mathbb{R}^d),
	\end{split}
\end{equation*}
according to \eqref{expecsf} and \eqref{kernelAL}, by a simple function approximation, we obtain
\begin{equation}\label{calculation1}
	\begin{split}
		\mathbb{E}_LS_{(|f|,|f|)_L}&=\int_{\mathbb{R}^d}|f(x/L)|F_L(x-x)dx+\int_{\mathbb{R}^d}|f(x/L)|H_L(x-x)dx\\
		&=[F_L(0)+H_L(0)]L^d\int_{\mathbb{R}^d}|f(x)|dx\\
		&\leq CL^d\int_{\mathbb{R}^d}|f(x)|dx\\
		&=O(L^d).
	\end{split}
\end{equation}
Hence for fixed $L\geq0$, the linear statistics
\begin{equation*}
	\begin{split}
		S_{(f,-f)_L}(\xi)=\sum\limits_{x\in\xi\cap \mathbb{R}_1^d}f\big (\frac{x}{L}\big )-\sum\limits_{y\in\xi\cap \mathbb{R}_2^d}f\big (\frac{y}{L}\big )\,,\,\,\,\,\xi\in\Conf(\mathbb{R}_1^d\sqcup\mathbb{R}_2^d)=\Conf(\mathbb{R}^d\sqcup\mathbb{R}^d),
	\end{split}
\end{equation*}
as a random variable makes sense, and
\begin{equation}\label{calculation2}
	\begin{split}
		\mathbb{E}_LS_{(f,-f)_L}&=\int_{\mathbb{R}^d}f(x/L)F_L(x-x)dx-\int_{\mathbb{R}^d}f(x/L)H_L(x-x)dx\\
		&=[F_L(0)-H_L(0)]L^d\int_{\mathbb{R}^d}f(x)dx.
	\end{split}
\end{equation}
Through a simple function approximation, \eqref{J-Hermitian-kernel}, \eqref{varsf} and \eqref{kernelAL} yield that
\begin{equation*}
	\begin{split}
		\Var_LS_{(f,-f)_L}&=\int_{\mathbb{R}^d}f^2(x/L)F_L(x-x)dx+\int_{\mathbb{R}^d}f^2(x/L)H_L(x-x)dx\\
		&\quad-\int_{\mathbb{R}^d\times\mathbb{R}^d}f(x/L)f(y/L)F_L(x-y)F_L(y-x)dxdy\\
		&\quad-\int_{\mathbb{R}^d\times\mathbb{R}^d}f(x/L)f(y/L)H_L(x-y)H_L(y-x)dxdy\\
		&\quad+2\int_{\mathbb{R}^d\times\mathbb{R}^d}f(x/L)f(y/L)G_L(x-y)G_L(y-x)dxdy\\
		&=[F_L(0)+H_L(0)]L^d\int_{\mathbb{R}^d}f^2(x)dx\\
		&\quad-\int_{\mathbb{R}^d\times\mathbb{R}^d}f(x/L)f(y/L)|F_L(x-y)|^2dxdy\\
		&\quad-\int_{\mathbb{R}^d\times\mathbb{R}^d}f(x/L)f(y/L)|H_L(x-y)|^2dxdy\\
		&\quad-2\int_{\mathbb{R}^d\times\mathbb{R}^d}f(x/L)f(y/L)|G_L(x-y)|^2dxdy.
	\end{split}
\end{equation*}
Using the Fourier transform and Plancherel theorem, we get
\begin{equation*}
	\begin{split}
		\Var_LS_{(f,-f)_L}&=[F_L(0)+H_L(0)]L^d\int_{\mathbb{R}^d}f^2(x)dx\\
		&\quad-L^d\int_{\mathbb{R}^d}|(\mathscr{F}f)(x)|^2(\mathscr{F}|F_L|^2)(x/L)dx\\
		&\quad-L^d\int_{\mathbb{R}^d}|(\mathscr{F}f)(x)|^2(\mathscr{F}|H_L|^2)(x/L)dx\\
		&\quad-2L^d\int_{\mathbb{R}^d}|(\mathscr{F}f)(x)|^2(\mathscr{F}|G_L|^2)(x/L)dx\\
		&=\sigma_L^2L^d\int_{\mathbb{R}^d}f^2(x)dx\\
		&\quad+L^d\int_{\mathbb{R}^d}|(\mathscr{F}f)(x)|^2\Big\{\big[(\mathscr{F}|F_L|^2)(0)+(\mathscr{F}|H_L|^2)(0)+2(\mathscr{F}|G_L|^2)(0)\big]\\
		&\quad\qquad\quad\quad-\big[(\mathscr{F}|F_L|^2)(x/L)+(\mathscr{F}|H_L|^2)(x/L)+2(\mathscr{F}|G_L|^2)(x/L)\big]\Big\}dx\\
		&=(\sigma^2+o(1))L^d\int_{\mathbb{R}^d}f^2(x)dx+L^dr(L)\\
		&=\sigma^2L^d\int_{\mathbb{R}^d}f^2(x)dx+o(L^d)+L^dr(L).
	\end{split}
\end{equation*}

\begin{claim}\label{claim-5}
	Under the assumptions of Theorem~\ref{thm-3}, $r(L)=o(1)$ as $L\to+\infty$.
\end{claim}
It follows from Claim \ref{claim-5} that 
\begin{equation}\label{calculation3}
	\begin{split}
		\Var_LS_{(f,-f)_L}=\sigma^2L^d\int_{\mathbb{R}^d}f^2(x)dx+o(L^d)\,\,\,\text{as}\,\,\,L\to +\infty.
	\end{split}
\end{equation}
Hence if $f\in B_c(\mathbb{R}^d)$, together with \eqref{calculation1}, \eqref{calculation2}, \eqref{calculation3} and Theorem \ref{thm-1}, we have
\begin{equation*}
	\begin{split}
		\frac{\big[\sum_{x\in\xi\cap\mathbb{R}_1^d}f(\frac{x}{L})-\sum_{y\in\xi\cap\mathbb{R}_2^d}f(\frac{y}{L})\big]-\big[F_L(0)-H_L(0)\big]L^d\int_{\mathbb{R}^d}f(x)dx}{\sqrt{\sigma^2L^d\int_{\mathbb{R}^d}f^2(x)dx+o(L^d)}}
	\end{split}
\end{equation*}
converges in distribution to $N(0,1)$ as $L\to +\infty$. And then we can deduce that 
\begin{equation*}
	\begin{split}
		\frac{1}{\sigma L^{d/2}}\Big\{\big[\sum\limits_{x\in\xi\cap \mathbb{R}_1^d}f\big (\frac{x}{L}\big )-\sum\limits_{y\in\xi\cap \mathbb{R}_2^d}f\big (\frac{y}{L}\big )\big]-\big[F_L(0)-H_L(0)\big]L^d\int_{\mathbb{R}^d}f(x)dx\Big\}
	\end{split}
\end{equation*}
converges in distribution to $N(0,\int_{\mathbb{R}^d}f^2(x)dx)$ as $L\to +\infty$.

As for general real-valued function $f\in L^1(\mathbb{R}^d)\cap L^2(\mathbb{R}^d)$, we choose $f_n\in B_c(\mathbb{R}^d)$ such that $f_n$ approach to $f$ as $n\to\infty$ in $L^2$-sense. Observe that
\begin{equation*}
	\begin{split}
		&\quad\mathbb{E}_L\Big[\Big(\frac{S_{(f,-f)_L}-\mathbb{E}_LS_{(f,-f)_L}}{\sigma L^{d/2}}-\frac{S_{(f_n,-f_n)_L}-\mathbb{E}_LS_{(f_n,-f_n)_L}}{\sigma L^{d/2}}\Big)^2\,\Big]\\
		&=\frac{1}{\sigma^2L^d}\mathbb{E}_L\Big[\big(S_{(f-f_n,\,-(f-f_n))_L}-\mathbb{E}_LS_{(f-f_n,\,-(f-f_n))_L}\big)^2\,\Big]\\
		&=\frac{1}{\sigma^2L^d}\Var_LS_{(f-f_n,\,-(f-f_n))_L}\\
		&=\frac{1}{\sigma^2L^d}\Big(\sigma^2L^d\int_{\mathbb{R}^d}[f(x)-f_n(x)]^2dx+o(L^d)\Big)\\
		&=\int_{\mathbb{R}^d}[f(x)-f_n(x)]^2dx+o(1)
	\end{split}
\end{equation*}
can be arbitrarily small when $n$ and $L$ are sufficiently large. Since 
\begin{equation*}
	\begin{split}
		\frac{1}{\sigma L^{d/2}}\Big\{\big[\sum\limits_{x\in\xi\cap \mathbb{R}_1^d}f_n\big (\frac{x}{L}\big )-\sum\limits_{y\in\xi\cap \mathbb{R}_2^d}f_n\big (\frac{y}{L}\big )\big]-\big[F_L(0)-H_L(0)\big]L^d\int_{\mathbb{R}^d}f_n(x)dx\Big\}
	\end{split}
\end{equation*}
converges in distribution to $N(0,\int_{\mathbb{R}^d}f_n^2(x)dx)$ as $L\to +\infty$, and
\begin{equation*}
	\begin{split}
		\lim\limits_{n\to\infty}\int_{\mathbb{R}^d}f_n^2(x)dx=\int_{\mathbb{R}^d}f^2(x)dx,
	\end{split}
\end{equation*}
a simple analysis implies that
\begin{equation*}
	\begin{split}
		\frac{1}{\sigma L^{d/2}}\Big\{\big[\sum\limits_{x\in\xi\cap \mathbb{R}_1^d}f\big (\frac{x}{L}\big )-\sum\limits_{y\in\xi\cap \mathbb{R}_2^d}f\big (\frac{y}{L}\big )\big]-\big[F_L(0)-H_L(0)\big]L^d\int_{\mathbb{R}^d}f(x)dx\Big\}
	\end{split}
\end{equation*}
converges in distribution to $N(0,\int_{\mathbb{R}^d}f^2(x)dx)$ as $L\to +\infty$.

\begin{proof}[Proof of Claim \ref{claim-5}]
Denote $W_L=|F_L|^2+|H_L|^2+2|G_L|^2$, a simple calculation about Fourier transform gives that
\begin{equation*}
	\begin{split}
		(\mathscr{F}W_L)(x)=\int_{\mathbb{R}^d}\big[&(\mathscr{F}F_L)(t)(\mathscr{F}F_L)(x-t)+(\mathscr{F}H_L)(t)(\mathscr{F}H_L)(x-t)\\
		&+2(\mathscr{F}G_L)(t)\overline{(\mathscr{F}G_L)}(x-t)\big]dt.
	\end{split}
\end{equation*}
Hence by Cauchy-Schwarz inequality and Plancherel theorem, we obtain
\begin{equation}\label{W(L)-estimateion}
	\begin{split}
		|(\mathscr{F}W_L)(x)|&\leq\int_{\mathbb{R}^d}\big[|(\mathscr{F}F_L)(t)|^2+|(\mathscr{F}H_L)(t)|^2+2|(\mathscr{F}G_L)(t)|^2\big]dt\\
		&=\int_{\mathbb{R}^d}\big[|F_L(x)|^2+|H_L(x)|^2+2|G_L(x)|^2\big]dt\\
		&=F_L(0)+H_L(0)-\sigma_L^2\\
		&\leq F_L(0)+H_L(0)\\
		&\leq C.
	\end{split}
\end{equation}

Split $r(L)$ into two parts as follow:
\begin{equation*}
	\begin{split}
		r(L)&=\int_{\mathbb{R}^d}|(\mathscr{F}f)(x)|^2\big[(\mathscr{F}W_L)(0)-(\mathscr{F}W_L)(x/L)\big]dx\\
		&=\int_{|x|>\sqrt{\kappa_L}}|(\mathscr{F}f)(x)|^2\big[(\mathscr{F}W_L)(0)-(\mathscr{F}W_L)(x/L)\big]dx\\
		&\quad+\int_{|x|\leq\sqrt{\kappa_L}}|(\mathscr{F}f)(x)|^2\big[(\mathscr{F}W_L)(0)-(\mathscr{F}W_L)(x/L)\big]dx\\
		&=r_1(L)+r_2(L).
	\end{split}
\end{equation*}
It follows from \eqref{W(L)-estimateion} that
\begin{equation}\label{r1(L)}
	\begin{split}
		|r_1(L)|&\leq\int_{|x|>\sqrt{\kappa_L}}|(\mathscr{F}f)(x)|^2\big|(\mathscr{F}W_L)(0)-(\mathscr{F}W_L)(x/L)\big|dx\\
		&\leq2C\int_{|x|>\sqrt{\kappa_L}}|(\mathscr{F}f)(x)|^2dx\\
		&=o(1).
	\end{split}
\end{equation}
To deal with $r_2(L)$, note that $W_L$ is non-negative and estimate by \eqref{W(L)-estimateion}, we have
\begin{equation*}
	\begin{split}
		|(\mathscr{F}W_L)(0)-(\mathscr{F}W_L)(x/L)|&=\Big|\int_{\mathbb{R}^d}W_L(t)(1-e^{-2\pi it\cdot \frac{x}{L}})dt\Big|\\
		&\leq\int_{|t|>\frac{L}{\kappa_L}}W_L(t)\big|1-e^{-2\pi it\cdot \frac{x}{L}}\big|dt\\
		&\quad+\int_{|t|\leq\frac{L}{\kappa_L}}W_L(t)\big|1-e^{-2\pi it\cdot \frac{x}{L}}\big|dt\\
		&\leq2\int_{|t|>\frac{L}{\kappa_L}}W_L(t)dt+\frac{2\pi|x|}{\kappa_L}\int_{|t|\leq\frac{L}{\kappa_L}}W_L(t)dt\\
		&\leq2\int_{|t|>\frac{L}{\kappa_L}}W_L(t)dt+\frac{2\pi|x|}{\kappa_L}\int_{\mathbb{R}^d}W_L(t)dt\\
		&=2\int_{|t|>\frac{L}{\kappa_L}}W_L(t)dt+\frac{2\pi|x|}{\kappa_L}(\mathscr{F}W_L)(0)\\
		&\leq2\int_{|t|>\frac{L}{\kappa_L}}W_L(t)dt+\frac{2\pi C|x|}{\kappa_L}\,.
 	\end{split}
\end{equation*}
Therefore,
\begin{equation}\label{r2(L)}
	\begin{split}
		|r_2(L)|&\leq\int_{|x|\leq\sqrt{\kappa_L}}|(\mathscr{F}f)(x)|^2\big|(\mathscr{F}W_L)(0)-(\mathscr{F}W_L)(x/L)\big|dx\\
		&\leq\Big(2\int_{|t|>\frac{L}{\kappa_L}}W_L(t)dt+\frac{2\pi C}{\sqrt{\kappa_L}}\Big)\int_{|x|\leq\sqrt{\kappa_L}}|(\mathscr{F}f)(x)|^2dx\\
	    &\leq2\Big(\int_{|t|>\frac{L}{\kappa_L}}W_L(t)dt+\frac{\pi C}{\sqrt{\kappa_L}}\Big)\int_{\mathbb{R}^d}|(\mathscr{F}f)(x)|^2dx\\
	    &=o(1).
	\end{split}
\end{equation}
The estimations \eqref{r1(L)} and \eqref{r2(L)} imply $r(L)=o(1)$.

This completes the proof of Theorem~\ref{thm-3}.
\end{proof}

\end{document}